\documentclass[11pt,draft]{article}
\usepackage{amssymb}
\usepackage{amsmath}
\usepackage{dmvnbase}
\usepackage[dvips]{graphics}
\usepackage[left=3cm, right=3cm, top=2cm, bottom=2cm]{geometry}
\newtheorem{theor}{Theorem}
\newtheorem{lem}{Lemma}[section]
\newtheorem{prop}[lem]{Proposition}

\newtheorem{rem}{Remark}[section]

\numberwithin{equation}{section}

\def\bea{\begin{eqnarray}}
\def\eea{\end{eqnarray}}
\def\beq{\begin{equation}}
\def\eeq{\end{equation}}
\def\be{\beq\begin{array}{c}}
\def\ee{\end{array}\eeq}
\def\bse{\begin{subequations}}
\def\ese{\end{subequations}}

\def\pbea{\begin{eqnarray*}}
\def\peea{\end{eqnarray*}}
\def\pbeq{\begin{equation*}}
\def\peeq{\end{equation*}}
\def\pbe{\pbeq\begin{array}{c}}
\def\pee{\end{array}\peeq}
\def\pbse{\begin{subequations*}}
\def\pese{\end{subequations*}}
\def\a{\alpha}
\def\d{\delta}
\def\b{\beta}

\def\e{\varepsilon}
\def\D{\Delta}
\def\phi{\varphi}

\def\sltwo{U_q({\widehat \slg_2})}

\def\advadva{U_q(A_2^{(2)})}
\def\ACh{\Ac_{\mathcal Ch}}
\def\AD{\Ac_\Dc}
\def\AR{\Ac_\Rc}
\def\rf#1{~\!\!(\ref{#1})}
\def\rfs#1{~\!\!\ref{#1}}

\def\RR{{\mathcal R}}

\DeclareMathOperator*{\Sym}{\rm Sym}

\newenvironment{proof}{{\flushleft\it Proof:}} {\hfill$\square$ \\}

\begin{document}

\bigskip
\hfill{ITEP-TH-65/09}

\begin{center}
{\Large\bf Three realizations of quantum \\ affine algebra $\advadva$}
\end{center}
\bigskip
\begin{center}
{\bf
A. Shapiro$^\star$\footnote{E-mail: alexander.m.shapiro@gmail.com}}\\
\bigskip
{$^\star$\it Institute of Theoretical \& Experimental Physics, 117259
Moscow, Russia}\\
\bigskip
\bigskip

\end{center}

\begin{abstract}
We establish explicit isomorphisms between three realizations of the quantum twisted affine algebra $\advadva$: the Drinfeld (``current'') realization, the Chevalley realization, and the so-called $RLL$ realization, investigated by Faddeev, Reshetikhin and Takhtajan.
\end{abstract}

\section{Introduction}

There exist just two quantum affine algebras of rank $2$: $\sltwo$ and $\advadva$. The algebra $\sltwo$ and its representation theory are both very well studied. The situation with $\advadva$ is somewhat different, although the algebra itself and its representation are interesting from the standpoints of both mathematics and mathematical physics.

The algebra $\advadva$ first appeared in works on physics. The fundamental representation of the $\Rc$-matrix of $\advadva$ was obtained in \cite{IK} as the $\Rc$-matrix of the quantum version of the Shabat-Mikhailov model, also known as the Izergin-Korepin model. The same representation was later obtained in \cite{J} among fundamental representations of $\Rc$-matrices of other nonexceptional affine Lie algebras. The Bethe ansatz technique was later extended onto $\advadva$ in \cite{T}. Finally, the ideas of the thermodynamic Bethe ansatz for $\advadva$ were developed in \cite{FRS}, and some finite dimensional representations were obtained there.

The algebra $\advadva$ was also investigated from an algebraic standpoint. In \cite{KT}, the Cartan-Weyl basis for $\advadva$ was established, and the universal $\Rc$-matrix was written in terms of infinite products of elements of the Cartan-Weyl basis. A classification of finite dimensional representations of algebra $\advadva$ was obtained in \cite{CP} by means of the Drinfeld polynomials. Later, an integral formula for the universal $\Rc$-matrix for $\advadva$ appeared in \cite{DK}, where $\advadva$ was treated as a topological Hopf algebra with the Drinfeld comultiplication. In the same work, the Serre relations in $\advadva$ were represented in terms of zeros and poles of products of the Drinfeld currents. Finally, integral representations for factors of the universal $\Rc$-matrix for $\advadva$ were derived in \cite{KS}, where $\advadva$ was endowed with the standard Hopf structure.

Quantum affine algebras allow three different realizations with different Hopf structures. The first one is the ``standard'' realization given by the Chevalley generators and relations determined by the corresponding Cartan matrix. The standard realization has a small number of generators, but is unfortunately very difficult to use in applications. The second realization, the Drinfeld ``new realization'', was first established in \cite{D} by means of generating functions (the Drinfeld ``currents'') and relations on them. The Drinfeld realization allows using methods of complex analysis. Moreover, it makes it possible to give the classification of the finite dimensional representations of the quantum affine algebra. Finally, the third one is the $RLL$ realization, based on the Faddeev-Reshetikhin-Takhtajan-Semenov-Tian-Shansky approach, where generators are combined into $L$-operators, satisfying the famous Yang-Baxter equation (see \cite{FRT}, \cite{RS}). The simplicity of comultiplication in the $RLL$ realization allows to construct new representations as tensor products of already known ones. The $RLL$ realization is therefore widely used in physical models.

It is universally acknowledged that the three realizations are isomorphic, although precise proofs hardly exist for any algebra other then of the $U_q(\widehat sl_n)$ type. For $U_q(\widehat sl_n)$, an isomorphism between the standard and the Drinfeld realizations was established in \cite{DF}, whereas links between the Drinfeld and the $RLL$ approaches were studied in \cite{DK2}.

In the case of quantum twisted affine algebras, the $RLL$ realization requires additional relations. Although it is believed that the three realizations are also isomorphic for the twisted algebras, there is no full understanding of what the exact isomorphism should look like. An isomorphism between the standard and the Drinfeld realizations for $\advadva$ was established in \cite{KS}, and an isomorphism between the $RLL$ and the Drinfeld realizations was partly derived in \cite{YZ}. A more complete bibliography can be found in \cite{H}.

In this work, we obtain a full description of the three realizations of algebra $\advadva$ (without grading element and with zero central charge), the isomorphisms between them, and the links between the three Hopf structures. We also pay special attention to the following fact: each realization has a ``minimal'' set of generators (or almost minimal in the twisted case) and an extended set of generators. These two sets should be linked by an analogue of the PBW-theorem if there exists one. In the standard realization, these sets are the Chevalley generators and the Cartan-Weyl basis. In the Drinfeld realization, they are the Drinfeld currents and the so-called ``composite currents'' (see \cite{DK3}). Finally, in the $RLL$-realization, the Gaussian coordinates immediately above or below the diagonal form the ``minimal'' set, and all the Gaussian coordinates form the extended set. We note that these extended sets of generators are crucial for calculations of the universal weight function. Here, we obtain a link between projections of the composite currents and the Gaussian coordinates for $\advadva$ as was done for $U_q(\widehat\glg_n)$ in \cite{KP}.

The paper is organized as follows. In the first section we introduce the three realizations. In Section $2$, we obtain the isomorphisms between them. In Section $3$, we represent elements of the extended set of Gaussian coordinates in terms of composite currents.

\section{Realizations of $\advadva$}

{\sl For simplicity of exposition, we consider the algebra $\advadva$ without grading element and with zero central charge. All the statements also hold for an algebra with grading element and arbitrary central charge, but the formulas become more involved.}

\subsection{Drinfeld realization}

Let $\AD$ be the associative algebra generated by elements
$$
e_n, f_n,\;\, n\in\Z,\quad a_n,\;\, n\in\Z\setminus\hc{0},\quad {\rm and} \quad k^{\pm 1},
$$
subject to certain commutation relations. The relations are given as formal power series identities for generating functions (currents)
$$
e(z)=\sum_{k\in\Z}e_kz^{-k},\quad f(z)=\sum_{k\in\Z}f_kz^{-k},
\quad K^\pm(z)=k^{\pm 1}\exp\left(\pm(q-q^{-1})\sum_{n>0}a_{\pm n}z^{\mp n}\right)
$$
as follows:
\pbea
(z-q^2w)(qz+w)e(z)e(w) &=& (q^2z-w)(z+qw)e(w)e(z),                           \\
(q^2z-w)(z+qw)f(z)f(w) &=& (z-q^2w)(qz+w)f(w)f(z),                           \\
K^+(z)e(w)K^+(z)^{-1} &=& \a(w/z)e(w),                                       \\
K^+(z)f(w)K^+(z)^{-1} &=& \a(w/z)^{-1}f(w),                                  \\
K^-(z)e(w)K^-(z)^{-1} &=& \a(z/w)^{-1}e(w),                                  \\
K^-(z)f(w)K^-(z)^{-1} &=& \a(z/w)f(w),                                       \\
K^\pm(z)K^\pm(w) &=& K^\pm(w)K^\pm(z),                                       \\
K^-(z)K^+(w) &=& K^+(w)K^-(z),                                               \\
e(z)f(w)-f(w)e(z) &=& \frac1{q-q^{-1}}\hr{\d(z/w)K^+(w)-\d(z/w)K^-(z)}.
\peea
where
$$
\a(x)=\dfrac{(q^2-x)(q^{-1}+x)}{(1-q^2x)(1+q^{-1}x)},
$$
and $\d\left(\dfrac zw\right)$ is a formal Laurent series, given by
$$
\d(z/w) = \sum\limits_{n\in\Z}(z/w)^n.
$$

The generating functions $e(z)$ and $f(z)$ also satisfy the cubic Serre relations (see \cite{D}):
\pbea
\Sym_{z_1,z_2,z_3} \hr{q^{-3}z_1 - (q^{-2}+q^{-1})z_2 + z_3} e(z_1)e(z_2)e(z_3) &=& 0, \\
\Sym_{z_1,z_2,z_3} \hr{q^{-3}z_1^{-1} - (q^{-2}+q^{-1})z_2^{-1} + z_3^{-1}} f(z_1)f(z_2)f(z_3) &=& 0, \\
\Sym_{z_1,z_2,z_3} \hr{q^3 z_1^{-1} - (q^{2}+q^{})z_2^{-1} + z_3^{-1}} e(z_1)e(z_2)e(z_3) &=& 0, \\
\Sym_{z_1,z_2,z_3} \hr{q^3 z_1 - (q^{2}+q^{})z_2 + z_3} f(z_1)f(z_2)f(z_3) &=& 0.
\peea

The Hopf algebra structure on $\AD$ can be defined as follows:
\pbea
\D_\Dc(e(z)) &=& e(z)\otimes 1+K^-(z)\otimes e(z), \\
\D_\Dc(f(z)) &=& 1\otimes f(z)+f(z)\otimes K^+(z), \\
\D_\Dc(K^\pm(z)) &=& K^\pm(z)\otimes K^\pm(z),     \\
S_\Dc(e(z)) &=& -(K^-(z))^{-1}e(z),                \\
S_\Dc(f(z)) &=& -f(z)(K^+(z))^{-1},                \\
S_\Dc(K^\pm(z)) &=& (K^\pm(z))^{-1},               \\
\e_\Dc(e(z)) &=& 0,                                \\
\e_\Dc(f(z)) &=& 0,                                \\
\e_\Dc(K^\pm(z)) &=& 1.
\peea
where $\D_\Dc$, $\e_\Dc$ and $S_\Dc$ are the comultiplication, the counit and the antipode maps, respectively. We call $\D_\Dc$ a \emph{Drinfeld comultiplication}. Here, we must note that $\AD$ is a topological bialgebra and $\D_\Dc$ is a map from $\AD$ to a topological extension of its tensor square (see \cite[Section 2]{EKP} for details).

\subsection{Chevalley realization}

Another realization of $\advadva$ is given in terms of Chevalley generators. Let the algebra $\ACh$ be the associative algebra generated by elements $e_{\pm\a}$, $e_{\pm(\d-2\a)}$, $k^{\pm1}_{\a}$, $k^{\pm1}_{\d-2\a}$, satisfying the following relations:
\pbea
k_{\a}e_{\pm\a}k^{-1}_{\a} = q^{\pm1}e_{\pm\a}, &\quad&
k_{\a}e_{\pm(\d-2\a)}k^{-1}_{\a} = q^{\mp 2}e_{\pm(\d-2\a)},
\\
k_{\d-2\a}e_{\pm{\a}}k^{-1}_{\d-2\a} = q^{\mp 2}e_{\pm{\a}}, &\quad&
k_{\d-2\a}e_{\pm{(\d-2\a)}}k^{-1}_{\d-2\a} = q^{\pm4}e_{\pm{(\d-2\a)}},
\\
k_{\a}^2 k_{\d-2\a} = 1, &\quad&
\hs{e_{\pm\a},e_{\mp(\d-2\a)}} = 0,
\\
\hs{e_\a,e_{-\a}} = \dfrac{k_\a-k_\a^{-1}}{q-q^{-1}}, &\quad&
\hs{e_{\d-2\a},e_{-(\d-2\a)}} = \dfrac{k_{\d-2\a}-k_{\d-2\a}^{-1}}{q-q^{-1}},
\\
(\ad_{q}e_{\pm\a})^5 e_{\pm(\d-2\a)} = 0, &\quad&
(\ad_{q}e_{\pm(\d-2\a)})^2 e_{\pm\a} = 0,
\peea
where
\pbea
(\ad_{q}e_{\pm\a})(x) &=& e_{\pm\a}x - k_\a^{\pm 1}xk_\a^{\mp 1}e_{\pm\a},
\\
(\ad_{q}e_{\pm(\d-2\a)})(x) &=& e_{\pm(\d-2\a)}x - k_{\d-2\a}^{\pm 1}xk_{\d-2\a}^{\mp 1}e_{\pm(\d-2\a)}.
\peea

The Hopf algebra structure associated with this realization can be defined by
\pbea
\D(e_{\a})=e_{\a}\otimes 1+k_{\a}\otimes e_{\a},
&\quad&
\D(e_{\d-2\a})=e_{\d-2\a}\otimes 1+k_{\d-2\a}\otimes e_{\d-2\a},
\\
\D(e_{-\a})=1\otimes e_{-\a}+e_{-\a}\otimes k^{-1}_{\a},
&\quad&
\D(e_{-(\d-2\a)})=1\otimes e_{-(\d-2\a)}+e_{-(\d-2\a)}\otimes k^{-1}_{\d-2\a},
\\
\D(k_{\a})\,=\,k_{\a}\otimes k_{\a},
&\quad&
\D(k_{\d-2\a})\,=\,k_{\d-2\a}\otimes k_{\d-2\a},
\\
\e(e_{\pm\a})=0,
&\quad&
\e(e_{\pm(\d-2\a)})=0,
\\
\e(k^{\pm1}_{\a})=1,
&\quad&
\e(k^{\pm1}_{\d-2\a})=1,
\\
S(e_{\a})=-k^{-1}_{\a}e_{\a},
&\quad&
S(e_{\d-2\a})=-k^{-1}_{\d-2\a}e_{\d-2\a},
\\
S(e_{-\a})=-e_{-\a}k_{\a},
&\quad&
S(e_{-(\d-2\a)})=-e_{-(\d-2\a)}k_{\d-2\a},
\\
S(k^{\pm1}_{\a}) =  k^{\mp1}_{\a},
&\quad&
S(k^{\pm1}_{\d-2\a}) =  k^{\mp1}_{\d-2\a},
\peea
where $\D$, $\e$ and $S$ are the comultiplication, the counit and the antipode maps, respectively. We call $\D$ the \emph{standard comultiplication}. We also consider the opposite comultiplication
$$
\D^{op}=\si\circ\D, \qquad\text{where}\qquad \si(u\otimes v) = v\otimes u.
$$

Now, let us recall the construction of the Cartan-Weyl basis, obtained for $\advadva$ in \cite{KT}\footnote{To adapt the results of \cite{KT} to our case, one should replace $q$-commutators with $q^{-1}$-commutators in the construction of the Cartan-Weyl basis.}. The twisted affine algebra $A_2^{(2)}$ has the root system $\D=\D_+\cup\D_-$, where
$$
\D_+ = \a\cup\hc{n\d,\;\pm\a+n\d,\;\pm2\a+(2n-1)\d\;|\;n\in\N}, \qquad \text{and} \qquad \D_- = -\D_+.
$$
The roots $\ga\in\D_+$ are called \emph{positive}, the roots $\ga\in\hc{n\d,\;n\in\Z}$ are called \emph{imaginary}, and the roots $\ga\in\D\setminus\hc{n\d,\;n\in\Z}$ are called \emph{real}. We define a scalar product on $\D$ by
$$
\hr{\a,\a} = 1, \qquad \hr{\a,\d} = \hr{\d,\d} = 0.
$$
Now, let
$$
q_\ga = q^{(\ga,\ga)}, \qquad (a)_q = \frac{q^a-1}{q-1}, \qquad (n)_q! = (1)_q\cdot\dots\cdot(n)_q,
$$
and
$$
\exp_q(x) = \sums{n\ge0}\frac{x^n}{(n)_q!} = 1+x+\frac{x^2}{(2)_q!}+\dots\,.
$$

We fix the normal ordering on $\D$
$$
\a,\; 2\a+\d,\; \a+\d, \; 2\a+3\d, \; \a+2\d, \; \dots,
\; \d, \; 2\d, \; 3\d, \; \dots, \;
2\d-\a, \; 3\d-2\a, \; \d-\a, \; \d-2\a.
$$
In accordance with the procedure of the construction of the Cartan-Weyl basis described in \cite{KT}, we set
\beq
\begin{split}
&e_{\d-\a}=a[e_{\a},e_{\d-2\a}]_{q^{-1}}, \qquad
e'_{\d}=b[e_{\a},e_{\d-\a}]_{q^{-1}}, \qquad
e_{\a+n\d}=b[e_{\a+(n-1)\d},e'_{\d}]_{q^{-1}}, \\
&e_{\d-\a+n\d}=b[e'_{\d},e_{\d-\a+(n-1)\d}]_{q^{-1}}, \qquad
e'_{n\d}=b[e_{\a+(n-1)\d},e_{\d-\a}]_{q^{-1}}, \\
&e_{2\a+(2n+1)\d}=a[e_{\a+n\d},e_{\a+(n+1)\d}]_{q^{-1}}, \qquad
e_{\d-2\a+2n\d}=a[e_{\d-\a+n\d}, e_{\d-\a+(n-1)\d}]_{q^{-1}},
\end{split}
\eeq
where
$$ \label{ab}
a=\frac{1}{\sqrt{q+q^{-1}}}, \qquad b=\frac{1}{\sqrt{q+1+q^{-1}}},
$$
and
$$
[e_\b,e_\ga]_{q^{-1}} = e_\b e_\ga-q^{-\hr{\b,\ga}}e_\ga e_\b.
$$
holds for all $\b, \ga \in \D$. Finally, we define \emph{imaginary} roots $e_{n\d}$ by means of the Schur polynomials:
\beq
e'_{n\d}=\sum_{p_{1}+2p_{2}+\dots +np_{n}=n}
\frac{\hr{(q-q^{-1})b^{-1}}^{\sum p_{i}-1}}{p_{1}!\ldots p_{n}!}
e_{\d}^{p_1}e_{2\d}^{p_2}\dots e_{n\d}^{p_n}.
\eeq
Monomials in the elements $e_{n\d}$, $e_{n\d\pm\a}$ and $k_\a^{\pm1}$ form a linear basis of $\ACh$.

\subsection{$RLL$ realization}

Let
\beq \label{R}
R(x) =
\begin{pmatrix}
1 & 0 & 0 & 0 & 0 & 0 & 0 & 0 & 0 \\
0 & a & 0 & d & 0 & 0 & 0 & 0 & 0 \\
0 & 0 & b & 0 & e & 0 & p & 0 & 0 \\
0 & f & 0 & a & 0 & 0 & 0 & 0 & 0 \\
0 & 0 & g & 0 & c & 0 & e & 0 & 0 \\
0 & 0 & 0 & 0 & 0 & a & 0 & d & 0 \\
0 & 0 & q & 0 & g & 0 & b & 0 & 0 \\
0 & 0 & 0 & 0 & 0 & f & 0 & a & 0 \\
0 & 0 & 0 & 0 & 0 & 0 & 0 & 0 & 1
\end{pmatrix},
\eeq
where
\begin{align*}
a &= \frac{q(x-1)}{q^2x-1},         & b &= \frac{q^2(qx+1)(x-1)}{(q^2x-1)(q^3x+1)},       &
c &= \frac{q(x-1)}{q^2x-1} + \frac{(q^2-1)(q^3+1)x}{(q^2x-1)(q^3x+1)},                    \\
d &= \frac{q^2-1}{(q^2x-1)},        & e &= \frac{(q^2-1)(x-1)q^{1/2}}{(q^2x-1)(q^3x+1)},  &
p &= \frac{(q^2-1)(q^3x+qx-q+1)}{(q^2x-1)(q^3x+1)},                                       \\
f &= \frac{(q^2-1)x}{(q^2x-1)},     & g &= \frac{(1-q^2)x(x-1)q^{5/2}}{(q^2x-1)(q^3x+1)}, &
q &= \frac{(q^2-1)x(q^3x-q^2x+q^2+1)}{(q^2x-1)(q^3x+1)}.
\end{align*}
We consider the generating functions
\pbeq
\begin{split}
l_{ij}^\pm(z) = \suml{n=0}{\infty}l_{ij}^\pm[\pm n]z^{\mp n}, &\qquad 1 \le i,j \le 3 \\
l_{ij}^+[0]=l_{ji}^-[0] = 0, &\qquad 1\le i<j \le 3
\end{split}
\peeq
and the shifted generating functions
$$
\tilde l_{ij}^\pm(z) = q^{\frac12(i-j)} l_{4-j,4-i}^\pm(-q^{-3}z).
$$
We define the \emph{$L$-operators} $L^\pm(z)$ and their shifts $\tilde L^\pm(z)$ as
$$
L^\pm(z) = \hr{l_{ij}^\pm(z)}_{i,j=1}^3, \qquad \tilde L^\pm(z) = \hr{\tilde l_{ij}^\pm(z)}_{i,j=1}^3.
$$
Finally, we recall the notion of $q$-determinant:
$$
\det\nolimits_q(L^\pm(z)) =
\sums{\tau\in \Sg_3}(-q)^{\sgn(\tau)}l^\pm_{1\tau(1)}(z)l^\pm_{2\tau(2)}(q^2z)l^\pm_{3\tau(3)}(q^4z),
$$
where $\Sg_n$ is the permutation group on $n$ elements.

Define $\AR$ as the associative algebra with the generators $l_{ij}^\pm[\pm n],\,n\in\N,\,1\le i,j \le 3$ and $l_{ij}^+[0]$, $l_{ji}^-[0]$ for $1\le j\le i\le 3$, subject to the following relations:
\beq \label{q-det}
\det\nolimits_q(L^\pm(z)) = 1,
\eeq
\beq \label{symmetry}
L^\pm(z)\tilde L^\pm(z) = I_3,
\eeq
\beq \label{pm_identity}
l_{ii}^+[0]l_{ii}^-[0] = l_{ii}^-[0]l_{ii}^+[0] = 1,
\eeq
\beq
\begin{split} \label{RLL}
R\hr{\frac zw}L_1^\pm(z)L_2^\pm(w) &= L_2^\pm(w)L_1^\pm(z)R\hr{\frac zw}, \\
R\hr{\frac zw}L_1^+(z)L_2^-(w) &= L_2^-(w)L_1^+(z)R\hr{\frac zw}.
\end{split}
\eeq
where $L_1(z) = L(z)\otimes I_3$, $L_2(z) = I_3 \otimes L(z)$, and $I_n$ denotes the $n\times n$ identity matrix.

The Hopf algebra structure on $\AR$ is given by
$$
\D_R(l_{ij}^\pm) = \suml{k=1}{3}l_{ik}^\pm \otimes l_{kj}^\pm, \qquad
S_R(L^\pm) = (L^\pm)^{-1}, \qquad
\e_R(L^\pm) = I_3,
$$
where $\D_R$, $S_R$ and $\e_R$ are the comultiplication, the antipode and the counit maps, respectively.

The $L$-operators admit the Gaussian decomposition
\beq \label{Gauss}
L^\pm(z) =
\begin{pmatrix}
1 & f^\pm_1(z) & f^\pm_{13}(z) \\
0 & 1          & f^\pm_2(z)    \\
0 & 0          & 1
\end{pmatrix}
\begin{pmatrix}
k^\pm_1(z) & 0          & 0          \\
0          & k^\pm_2(z) & 0          \\
0          & 0          & k^\pm_3(z)
\end{pmatrix}
\begin{pmatrix}
1             & 0          & 0 \\
e^\pm_1(z)    & 1          & 0 \\
e^\pm_{31}(z) & e^\pm_2(z) & 1
\end{pmatrix},
\eeq
where
\begin{align*}
e^\pm_i(z) &= \sums{n\ge0}e^\pm_i[n]z^{\mp n},\quad i=1,2,   &e^\pm_{31}(z) &= \sums{n\ge0}e^\pm_{31}[n]z^{\mp n}, \\
f^\pm_i(z) &= \sums{n\ge0}f^\pm_i[n]z^{\mp n},\quad i=1,2,   &f^\pm_{13}(z) &= \sums{n\ge0}f^\pm_{13}[n]z^{\mp n}, \\
k^\pm_i(z) &= \sums{n\ge0}k_i^\pm[n]z^{\mp n},\quad i=1,2,3.
\end{align*}
and
$$
e^-_1[0] = e^-_2[0] = e^-_{31}[0] = 0, \qquad f^+_1[0] = f^+_2[0] = f^+_{13}[0] = 0.
$$
We have thus defined the Gaussian generators
$$
f^\pm_1[n],\;f^\pm_2[n],\;f^\pm_{13}[n], \qquad
e^\pm_1[n],\;e^\pm_2[n],\;e^\pm_{31}[n], \qquad
k^\pm_1[n],\;k^\pm_2[n],\;k^\pm_3[n], \quad n\ge0,
$$
subject to relations \rf{q-det} -- \rf{RLL}. Here $f^\pm_1[n],\;f^\pm_2[n],\;e^\pm_1[n],\;e^\pm_2[n]$ form the ``minimal'' set of generators, while $f^\pm_{13}[n]$ and $e^\pm_{31}[n]$ form the extended one.

\section{Links between realizations}

\subsection{Universal $\Rc$-matrix}

Recall that the universal $\Rc$-matrix of a quasitriangular Hopf algebra $\Ac$ is an element $\Rc\in\Ac\otimes\Ac$, subject to the following relations:
\beq \label{intertwining}
\D^{op}(x)=\Rc\D(x)\Rc^{-1}
\eeq
for any $x\in\Ac$, and
\beq \label{quasi}
(\D\otimes\id)\Rc=\Rc_{13}\Rc_{23}, \qquad (\id\otimes\Delta)\Rc=\Rc_{13}\Rc_{12},
\eeq
where
$$
\Rc_{12}=\Rc\otimes1, \qquad \Rc_{23}=1\otimes \Rc, \qquad \Rc_{13} = (\si\otimes\,\id)(\Rc_{23})
$$
are elements of $\Ac\otimes\Ac\otimes\Ac$.

Now we fix the representation $\pi_z$ defined on Chevalley generators as
\begin{align*}
\pi_{z}\hr{e_\a} &=
\begin{pmatrix}
0&q^\frac14&0 \\
0&0&-q^{-\frac14} \\
0&0&0
\end{pmatrix},
&
\pi_{z}\hr{e_{\d-2\a}} &=
\begin{pmatrix}
0&0&0 \\
0&0&0 \\
\sqrt{q+q^{-1}}&0&0
\end{pmatrix}z, \\
\pi_{z}\hr{e_{-\a}} &=
\begin{pmatrix}
0&0&0 \\
q^{-\frac14}&0&0 \\
0&-q^\frac14&0
\end{pmatrix},
&
\pi_{z}\hr{e_{-\d+2\a}} &=
\begin{pmatrix}
0&0&\sqrt{q+q^{-1}} \\
0&0&0 \\
0&0&0
\end{pmatrix}z^{-1}, \\
\pi_{z}\hr{k_\a} &=
\begin{pmatrix}
q&0&0 \\
0&1&0 \\
0&0&q^{-1}
\end{pmatrix},
&
\pi_{z}\hr{k_{\d-2\a}} &=
\begin{pmatrix}
q^{-2}&0&0 \\
0&1&0 \\
0&0&q^2
\end{pmatrix}.
\end{align*}
After the representation is chosen, we obtain the $R$-matrix
\footnote{
The scaling of the basis of the $\advadva$-module is chosen in such a way that $R$-matrix $R(x)$ becomes symmetric with respect to the secondary diagonal. In that basis, the $R$-matrix coincides with the one obtained in \cite{J}.
}
$$
R\hr{\frac zw} = \hr{\pi_z\otimes\pi_w}\Rc
$$
directly from the definition of $\Rc$ by applying $\pi_z\otimes\pi_w$ to relation \rf{intertwining} and substituting the Chevalley generators for $x$. Condition \rf{intertwining} defines $R(x)$ up to a scalar.

\subsection{Chevalley and Drinfeld realizations}

Following \cite{KT}, we represent the universal $\Rc$-matrix of $\ACh$ as
\footnote{The formula for the $\Rc$-matrix differs from the formula in \cite{KT} because $q$-commutators are replaced with $q^{-1}$-commutators. Moreover, our formula differs from the one in \cite{KS}; the $\Rc$-matrix in \cite{KS} is inverse to the $\Rc$-matrix here.}:
\beq \label{R_matrix}
\begin{split}
\Rc = \Rc_-\Kc q^{-h \otimes h}\Rc_+, &\qquad
\Rc_- = \hr{\prodl{\ga<\d}\ra \exp_{q_\ga}
\hr{(q^{-1}-q)e_\ga \otimes e_{-\ga}}}, \\
\Kc = \exp\hr{-\sums{n>0}{S_{n}}}, &\qquad
\Rc_+ = \hr{\prodl{\ga>\d}\ot\exp_{q_\ga}
\hr{(q^{-1}-q)e_\ga k_\a^{-(\a,\ga)}\otimes k_\a^{(\a,\ga)} e_{-\ga}}},
\end{split}
\eeq
where $\ga$ ranges over real positive roots, $h$ is defined by
$$
q^{\pm h} = k^{\pm 1},
$$
and $S_{n}$ is given by the formula
$$
S_{n}=\frac{n\hr{q-q^{-1}}^2\hr{q+1+q^{-1}}(e_{n\d}\otimes e_{-n\d})}
{\hr{q^n-q^{-n}}\hr{q^{n}+(-1)^{n+1}+q^{-n}}}.
$$

Now, we can find the scalar prefactor for $R(x)$. Applying $\pi_z \otimes \pi_w$ to $\Kc \cdot q^{-h \otimes h}$, we obtain the scalar
\beq \label{prefactor}
q^{-1}\exp\hr{-\sums{n>0}\frac{\hr{q^n-q^{-n}}x^n}{\hr{q^n+(-1)^{n+1}+q^{-n}}n}}.
\eeq
Therefore, $R(x)$ is equal to matrix \rf{R} times scalar prefactor \rf{prefactor}.

\begin{rem}
Scalar \rf{prefactor} expands into the product
$$
q^{-1}\prodl{k=0}{\infty}\hr{\frac{(1+q^{3k+2}x)(1-q^{3k+3}x)}{(1+q^{3k}x)(1-q^{3k+1}x)}}^{(-1)^k},
$$
in the domain $|q|<1$ and into the product
$$
q^{-1}\prodl{k=0}{\infty}\hr{\frac{(1+q^{-3k}x)(1-q^{-(3k+1)}x)}{(1+q^{-(3k+2)}x)(1-q^{-(3k+3)}x)}}^{(-1)^k},
$$
in the domain $|q|>1$.
\end{rem}

\begin{theor} \label{theor Ch-Dr}
The isomorphism between associative algebras $\ACh$ and $\AD$ is established by the following map:
\beq \label{Ch-Dr}
\begin{array}{rclrcl}
k_{\d-2\a}  &\mapsto& k^{-2},                             &\qquad k_{\a} &\mapsto& k,     \\
e_{\d-2\a}  &\mapsto& a(qf_1f_0-f_0f_1)k^{-2},            &\qquad e_{\a} &\mapsto& e_{0},   \\
e_{-\d+2\a} &\mapsto& ak^2(q^{-1}e_0e_{-1}-e_{-1}e_0),    &\qquad e_{-\a} &\mapsto& f_{0},
\end{array}
\eeq
where $a$ is defined in \rf{ab}.
\end{theor}
\begin{proof}
A straightforward verification and the inductive construction of the Cartan-Weyl basis show that the map \rf{Ch-Dr} above is a surjective homomorphism. More precisely, every Drinfeld generator can be obtained as
\beq \label{CW-Dr}
\begin{array}{rcllrcll}
e_{\a+n\d}      &\mapsto& e_n,    &n\ge0,    & \qquad -k_\a^{-1}e_{\a-n\d} &\mapsto& e_{-n}, &n>0,                 \\
e_{-\a-n\d}     &\mapsto& f_{-n}, &n\ge0,    & \qquad -e_{-\a+n\d}k_\a     &\mapsto& f_n,    &n>0,                 \\
k_\a            &\mapsto& k,      &          & \qquad -b^{-1}e_{n\d}       &\mapsto& a_n,    &n\in\Z\setminus\{0\}.
\end{array}
\eeq
Reversing all the arrows in map \rf{CW-Dr}, we obtain a map inverse to \rf{Ch-Dr}. Finally, map \rf{CW-Dr} with reversed arrows is also an epimorphism, and the statement of the theorem follows.
\end{proof}
\begin{rem} \label{rem} The above map is an isomorphism of associative algebras. It preserves the counit but does not respect comultiplication maps $\D$ and $\D_\Dc$.
\end{rem}
As for the link between the Drinfeld and the standard comultiplications, we have the following
\begin{prop}
The tensor $\RR_-$ is a cocycle for $\D_\Dc$ so that for any $x\in\ACh$
$$
\D(x)=\hr{\Rc_-}^{-1}\D_\Dc(x)\Rc_-.
$$
\end{prop}
The proof is given in \cite{KS}. Another expression for $\Rc_-$ in terms of the Drinfeld generators can also be found there.

\subsection{From Chevalley to $RLL$ realization}

According to \cite{D}, the universal $R$-matrix of any quasitriangular Hopf algebra satisfies the quantum Yang-Baxter equation
\beq \label{qybe}
\Rc_{12}\Rc_{13}\Rc_{23}=\Rc_{23}\Rc_{13}\Rc_{12}.
\eeq
\begin{theor} \label{theor RLL-Ch}
An isomorphism between the associative algebras $\AR$ and $\ACh$ can is given by the maps
\beq \label{Loper}
\begin{split}
L^-(z) &\mapsto (\pi_z\otimes\id)\Rc, \\
L^+(z) &\mapsto (\pi_z\otimes\id)\Rc_{21}^{-1}.
\end{split}
\eeq
\end{theor}

\begin{proof}
Applying $\;\pi_z\otimes\pi_w\otimes\id\;$, $\;\pi_z\otimes\id\otimes\pi_w\;$ and $\;\id\otimes\pi_z\otimes\pi_w\;$ to relation \rf{qybe}, we derive commutation relations \rf{pm_identity} and \rf{RLL}. Using formulas \rf{R_matrix} and \rf{Loper}, we obtain the Gaussian decomposition \rf{Gauss}.

The matrix $R(-q^3)$ has a one-dimensional kernel, which is spanned by vector
$$
v = q^{-\frac12} v_1\otimes v_3 + v_2\otimes v_2 + q^\frac12 v_3\otimes v_1,
$$
where $\hc{v_1,v_2,v_3}$ is the basis of the representation $\pi_z$. It follows that $v$ is an eigenvector of operator $L_1(z)L_2(-q^{-3}z)$. Let $\la(z)$ be an eigenvalue of $v$. One can check that $\hr{\pi_z\otimes\pi_{-q^{-3}z}}\D(x)$ vanishes on $v$ for every $x\in\ACh$ (it is sufficient to varify the vanishing condition for the Chevalley generators only). Applying $\hr{\pi_z\otimes\pi_{-q^{-3}z}\otimes\id}$ to the first equation in \rf{quasi}, we obtain $\la(z)=1$. Hence, $v$ is stable under $L_1(z)L_2(-q^{-3}z)$, which turns out to be equivalent to the relation \rf{symmetry}:
$$
\begin{pmatrix}
l_{11}(z) & l_{12}(z) & l_{13}(z) \\
l_{21}(z) & l_{22}(z) & l_{23}(z) \\
l_{31}(z) & l_{32}(z) & l_{33}(z)
\end{pmatrix}
\begin{pmatrix}
          l_{33}(-q^{-3}z)  &  q^{-\frac12} l_{23}(-q^{-3}z)  &  q^{-1}       l_{13}(-q^{-3}z) \\
q^\frac12 l_{32}(-q^{-3}z)  &               l_{22}(-q^{-3}z)  &  q^{-\frac12} l_{12}(-q^{-3}z) \\
q         l_{31}(-q^{-3}z)  &  q^\frac12    l_{21}(-q^{-3}z)  &               l_{11}(-q^{-3}z)
\end{pmatrix}
=I_3.
$$
Next, it is possible to check that
$$
\det\nolimits_q(L^\pm(z)) = k_1^\pm(z)k_2^\pm(q^{-2}z)k_3^\pm(q^{-4}z).
$$
Following the procedure of the construction of the Cartan-Weyl basis we get
$$
\pi_z\hr{e_{\pm n\d}} = b\frac{q^n-q^{-n}}{n(q-q^{-1})}
\begin{pmatrix}
-z^{\pm n} & 0 & 0 \\
0 & (q^2z)^{\pm n} - (-qz)^{\pm n} & 0 \\
0 & 0 & (-q^3z)^{\pm n}
\end{pmatrix}, \quad n>0.
$$
Applying $\pi_z$ to either first or second tensor component of $\Kc$, we arrive at
\beq \label{add2}
\begin{split}
k_1^\pm(z) &= k_3^\pm(-q^{-3}z)^{-1}, \\
k_2^\pm(z) &= k_3^\pm(-q^{-1}z)k_3^\pm(q^{-2}z)^{-1},
\end{split}
\eeq
which implies the desired relation \rf{q-det}.

Since all the Chevalley generators are in the image of map \rf{Loper}, we obtain an epimorphism from $\AR$ to $\ACh$ (and, therefore, to $\AD$). In Theorem \rfs{theor Dr-RLL} below, we establish an epimorphism from $\AD$ to $\AR$, which is inverse to \rf{Loper}. Hence, the theorem is proved modulo the result of Theorem \rfs{theor Dr-RLL}.
\end{proof}

\begin{prop}
Let $\phi(x)\in\ACh$ be the image of $x\in\AR$ under isomorphism \rf{Loper}. Then
$$
\phi\hr{\D_R(x)} = \si\circ\D(\phi(x)).
$$
\end{prop}

\begin{proof}
The proof can be obtained by applying $\pi_z\otimes\id\otimes\id$ to the relation
$$
(\id\otimes\D^{op})\Rc = \Rc_{12}\Rc_{13},
$$
which is equivalent to \rf{quasi}.
\end{proof}

\subsection{From $RLL$ to Drinfeld realization}

First of all, we express all generators of the $RLL$ realization in terms of $e_1^\pm(z)$, $f_1^\pm(z)$, $k_1^\pm(z)$. The expressions
\beq \label{k2k3}
\begin{split}
k_2^\pm(z) &= k_1^\pm(-qz)k_1^\pm(q^2z)^{-1}, \\
k_3^\pm(z) &= k_1^\pm(-q^3z)^{-1}
\end{split}
\eeq
follow directly from \rf{add2}. Relations \rf{RLL} imply
\begin{multline*}
(q^2-1)z(q^3z+w)L_{23}^\pm(z)L_{33}^\pm(w) + q(z-w)(q^3z+w)L_{33}^\pm(z)L_{23}^\pm(w) = \\ = (q^2z-w)(q^3z+w)L_{23}^\pm(w)L_{33}^\pm(z),
\end{multline*}
which results in
$$
q(z-w)k_3^\pm(z)f_2^\pm(w) = (q^2z-w)f_2^\pm(w)k_3^\pm(z) - (q^2-1)zf_2^\pm(z)k_3^\pm(z).
$$
Setting $w=q^2z$, we get
$$
f_2^\pm(q^2z) = q^{-1}k_3^\pm(z)^{-1}f_2^\pm(z)k_3^\pm(z).
$$
On the other hand, relation \rf{symmetry} implies
$$
f_1^\pm(z) = -q^{-\frac12}k_3^\pm(-q^{-3}z)^{-1}f_2^\pm(-q^{-3}z)k_3^\pm(-q^{-3}z).
$$
We hence obtain
\beq \label{e2f2}
\begin{split}
f_2^\pm(z) &= -q^{-\frac12}f_1^\pm(-qz), \\
e_2^\pm(z) &= -q^{\frac12}e_1^\pm(-qz),
\end{split}
\eeq
where the expression for $e_2^\pm(z)$ is derived in the similar way.
The expressions for $e_{31}^\pm(z)$ and $f_{13}^\pm(z)$ are considered in Section $3$.

\begin{rem}
It is also possible to use the explicit formula for $L^\pm(z)$ to derive relations \rf{e2f2}, as we did for relations \rf{add2}.
\end{rem}

Now let
\pbeq
\begin{split}
E_i(z) &= e_i^+(z) - e_i^-(z), \\
F_i(z) &= f_i^+(z) - f_i^-(z).
\end{split}
\peeq
Then the equalities
\beq \label{E2F2}
\begin{split}
E_2(z) &= -q^\frac 12 E_1(-qz), \\
F_2(z) &= -q^{-\frac 12} F_1(-qz).
\end{split}
\eeq
hold. Moreover, relations \rf{q-det}--\rf{RLL} imply
\beq \label{kk}
\begin{split}
k_1^\pm(z)k_1^\pm(w) &= k_1^\pm(w)k_1^\pm(z), \\
k_1^-(z)k_1^+(w) &= k_1^+(w)k_1^-(z),
\end{split}
\eeq

\beq \label{kEkkFk}
\begin{split}
k_1^\pm(z)E_1(w)k_1^\pm(z)^{-1} &= \frac{q(z-w)}{q^2z-w}E_1(w), \\
k_1^\pm(z)^{-1}F_1(w)k_1^\pm(z) &= \frac{z-q^2w}{q(z-w)}F_1(w),
\end{split}
\eeq

\beq \label{EEFF}
\begin{split}
(qz+w)(z-q^2w)E_1(z)E_1(w) &= (z+qw)(q^2z-w)E_1(w)E_1(z), \\
(z+qw)(q^2z-w)F_1(z)F_1(w) &= (qz+w)(z-q^2w)F_1(w)F_1(z),
\end{split}
\eeq

\beq \label{EF}
\hs{E_1(z),F_1(w)} = (q-q^{-1})\d\hr{\frac zw}\hr{k_1^+(w)k_2^+(w)^{-1} - k_1^-(z)k_2^-(z)^{-1}}.
\eeq

\begin{theor} \label{theor Dr-RLL}
An isomorphism\footnote{As in remark \rf{rem}, the mapping above is an isomorphism of associative algebras. It preserves the counit but does not respect comultiplication maps $\D$ and $\D_\Dc$.} between the algebras $\AD$ and $\AR$ is given by the following maps:
\beq \label{Dr-RLL}
\begin{array}{lcl}
q^{-\frac 14} \hr{q-q^{-1}}e(qz)   &\mapsto &E_1(z),                       \\
q^{\frac 14}\hr{q-q^{-1}}f(qz)     &\mapsto &F_1(z),                       \\
K^\pm(qz)                          &\mapsto &k_1^\pm(z)k_2^\pm(z)^{-1}.
\end{array}
\eeq
\end{theor}

\begin{proof}
Formulas \rf{kk}--\rf{EF} imply that the map \rf{Dr-RLL} is a homomorphism. For all $n\ge0$ Gaussian coordinates $k_1^\pm[n]$ can be obtained step by step from the product $k_1^\pm(z)k_2^\pm(z)^{-1}$, which together with relations \rf{k2k3} and \rf{E2F2} imply that map \rf{Dr-RLL} is an epimorphism. Finally, one can check that the composition of maps \rf{Loper}, \rf{Ch-Dr} and \rf{Dr-RLL} is the identity map.
\end{proof}

\section{Link with composite Drinfeld currents}
Here we revise some definitions, given in \cite{KS} and then express the generating functions $f_{13}^\pm(z)$ and $e_{31}^\pm(z)$ in terms of the Drinfeld currents.

Recall that in any quantum affine algebra there exist two types of Borel subalgebras. Borel subalgebras of the first type come from the Drinfeld realization. Let $U_F$ denote the subalgebra of $\advadva$ generated by $k^{\pm 1},f_n,\,n\in\Z;\;a_n,\,n>0$, and let $U_E$ denote the subalgebra of $\advadva$ generated by $k^{\pm 1},e_n,\,n\in\Z;\;a_n,\,n<0$. The ``current'' Borel subalgebra $U_F$ contains subalgebra $U_f$ generated by $f_n,\,n\in\Z$. The ``current'' Borel subalgebra $U_E$ contains subalgebra $U_e$ generated by $e_n,\,n\in\Z$.

Borel subalgebras of the second type are obtained via the Chevalley realization. Let $U_q(\bgt_+)$ and $U_q(\bgt_-)$ denote a pair of subalgebras of $\advadva$ generated by
\pbe
e_{\a},\;\,e_{\d-2\a},\;\,k^{\pm1}_{\a} \qquad\text{and}\qquad
e_{-\a},\;\,e_{-(\d-2\a)},\;\,k^{\pm1}_{\a}
\pee
respectively. In terms of the Drinfeld realization, these subalgebras are generated by
\pbe
k^{\pm 1},\;\,e_0,\;\,qf_1f_0-f_0f_1 \qquad\text{and}\qquad k^{\pm 1},\;\,f_0,\;\,q^{-1}e_0e_{-1}-e_{-1}e_0.
\pee
respectively.

Let $U_F^+$, $U_f^-$, $U_e^+$ and $U_E^-$ denote the following intersections
of Borel subalgebras:

\pbea
U_f^-=U_F\cap U_q(\bgt_-), &\qquad U_F^+=U_F\cap U_q(\bgt_+), \\
U_e^+=U_E\cap U_q(\bgt_+), &\qquad U_E^-=U_E\cap U_q(\bgt_-).
\peea
The superscript sign indicates the Borel subalgebra $U_q(\bgt_\pm)$ containing the given algebra, and the subscript letter indicates ``current'' Borel subalgebra $U_F$ or $U_E$ that it is intersected with. These letters are capitals if the subalgebra contains imaginary root generators $a_n$ and are lower case otherwise.

Let $P^+$ and $P^-$ denote projection operators such that for any $f_+\in U_F^+$ and for any $f_-\in U_f^-$
\be\label{Pdef}
P^+(f_-f_+)=\e(f_-)f_+, \qquad P^-(f_-f_+)=f_-\e(f_+).
\ee
Also let $P^{*+}$ and $P^{*-}$ denote projecton operators such that for any $e_+\in U_e^+$ and for any $e_-\in U_E^-$
\be\label{Pdefa}
P^{*+}(e_+e_-)=e_+\e(e_-), \qquad P^{*-}(e_+e_-)=\e(e_+)e_-.
\ee

Finally, let us introduce composite currents $s(z)$ and $r(z)$ as follows:
\begin{align*}
s(z) &= \res\limits_{w=-q^{-1} z}f(z)f(w)\frac{dw}w, \\
r(z) &= \res\limits_{w=-qz}e(w)e(z)\frac{dw}w.
\end{align*}

\begin{rem}
The currents $s(z)$ and $\tilde s(z)$ were defined in \cite{KS}. The definition of $s(z)$ coincides with the one given above, while $\tilde s(z) = -s(-qz)$ or equivalently
$$
\tilde s(z) = \res\limits_{w=-qz}f(w)f(z)\frac{dw}w.
$$
An involution $\iota$ of the algebra $\advadva$ was also defined in \cite{KS}:
$$
\iota(e_n) = f_{-n}, \qquad \iota(f_n) = e_{-n}, \qquad \iota(a_n) = a_{-n}, \qquad \iota(K_0) = K_0^{-1}.
$$
Therefore, the current $r(z)$ satisfies the equality
$$
r(z) = \iota\hr{-\tilde s\hr{z^{-1}}}.
$$
\end{rem}

\begin{prop} \label{gauss_currents}
Isomorphism \rf{Dr-RLL} maps the Gaussian coordinates $e_{31}^\pm(z)$ and $f_{13}^\pm(z)$ to the expressions
\beq \label{e31f13}
\begin{split}
e_{31}^\pm(z) &\mapsto
\mp(1-q) \hr{ q^{-1} \hs{P^{*\pm}(e(-q^2z)),e_0}_{q^{-1}} + \hs{-q^2ze_{-1},P^{*\pm}(e(-q^2z))\mp e_0}_{q^{-1}} }, \\
f_{13}^\pm(z) &\mapsto
\pm(1-q) \hr{ q\hs{P^\pm(f(-q^2z)),f_0}_{q^{-1}} + \hs{(-q^2z)^{-1}f_1,P^\pm(f(-q^2z))\pm f_0}_{q^{-1}} }.
\end{split}
\eeq
\end{prop}

\begin{proof}
We consider only the case of $f_{13}^+(z)$ because the other cases are similar.
Relations \rf{RLL} imply
\begin{multline*}
(q^2-1)z(q^3z+w)l_{13}^+(z)l_{22}^-(w) + q(z-w)(q^3z+w)l_{23}^+(z)l_{12}^-(w) = \\ =
(q^2-1)w(q^3z+w)l_{13}^-(w)l_{22}^+(z) + q(z-w)(q^3z+w)l_{12}^-(w)l_{23}^+(z).
\end{multline*}
Assuming $w=0$, we get
$$
(q-q^{-1})l_{13}^+(z) = \hs{l_{12}^-(0),l_{23}^+(z)},
$$
or, equivalently,
$$
(q-q^{-1})f_{13}^+(z)k_3^+(z) = f_1^-[0]f_2^+(z)k_3^+(z) - f_2^+(z)k_3^+(z)f_1^-[0].
$$
Using \rf{RLL} once again, we obtain
$$
k_3^+(z)f_1^-[0] = qf_1^-[0]k_3^+(z) + (q-q^{-1})q^{\frac12}f_2^+(z)k_3^+(z),
$$
and therefore derive
$$
f_{13}^+(z) = \frac1{q-q^{-1}}\hr{q^\frac12 f_1^+(-qz)f_1^-[0] - q^{-\frac12}f_1^-[0]f_1^+(-qz)} - q^{-\frac12}f_1^+(-qz)^2.
$$

Now, let us recall a result from \cite{KS}: on one hand, we have
\begin{multline*}
P(f(z_1)f(z_2))= f^+(z_1)\hr{f^+(z_2) - \dfrac{q^2-1}{q^2-z_2/z_1}f^+(z_1)} - \\ -
\dfrac{(1+q^3)(1-z_2/z_1)}{(1+q)(q^2-z_2/z_1)(1+qz_2/z_1)}P(s(z_1)),
\end{multline*}
and, on the other hand,
$$
P(f(z_1)f(z_2)) = \frac{(q^2-z_1/z_2)(q^{-1}+z_1/z_2)}{(1-q^2z_1/z_2)(1+q^{-1}z_1/z_2)}P(f(z_2)f(z_1)).
$$
Equating the right-hand sides of these expressions, multiplying them by $z_2$, and letting $z_2$ tend to infinity, we obtain
$$
(q-q^{-1)}P(f(z))^2 = \hr{f_1P(f(z))-q^{-1}P(f(z))f_1}z^{-1} + \frac{1+q^3}{q(1+q)}P(s(z)) + \hr{f_1f_0-q^{-1}f_0f_1}z^{-1}.
$$
Finally, since isomorphism \rf{Dr-RLL} maps $f_1^\pm(-qz)$ to $\pm q^{\frac14}(q-q^{-1})P^\pm(f(-q^2z))$, we derive the desired formula \rf{e31f13}.
\end{proof}

\begin{theor} \label{theor}
Isomorphism \rf{Dr-RLL} links the Gaussian coordinates $f_{13}^+(z)$, $e_{31}^+(z)$ and the composite Drinfeld currents $s(z)$, $r(z)$ by the following expressions:
\beq
\begin{split}
P^+(s(-q^2z))    &= \frac{1}{(q-q^{-1})(q-1+q^{-1})}f_{13}^+(z), \\
P^-(s(-q^2z))    &= \frac{1}{(q-q^{-1})(q-1+q^{-1})}\hr{f_1^-(z)f_2^-(z) - f_{13}^-(z)}, \\
P^{*+}(r(-q^2z)) &= \frac{1}{(q-q^{-1})(q-1+q^{-1})}e_{31}^+(z), \\
P^{*-}(r(-q^2z)) &= \frac{1}{(q-q^{-1})(q-1+q^{-1})}\hr{e_2^-(z)e_1^-(z) - e_{31}^-(z)}.
\end{split}
\eeq
\end{theor}

\begin{proof}
The statement for $s(z)$ follows directly from Proposition \rfs{gauss_currents} and Theorem~1 in \cite{KS}. For $r(z)$ the involution $\iota$ should also be used.
\end{proof}

\begin{rem}
Since $f_1^-(z)f_2^-(z) - f_{13}^-(z)$ and $e_2^-(z)e_1^-(z) - e_{31}^-(z)$ are the corner elements in the Gaussian decomposition of the inverse $L$-operator $\hr{L^-(z)}^{-1}$, it might make sense to use the elements of the decomposition of the inverse $L$-operator $\hr{L^-(z)}^{-1}$ as the Gaussian coordinates instead of the elements of the decomposition of usual $L$-operator $L^-(z)$.
\end{rem}

\section*{Acknowledgements}
The author thanks Sergey Khoroshkin for the invaluable help and his constant interest to the work. The author is also grateful to Stanislav Pakuliak for many fruitful discussions. The author was supported by RFBR grant 08-01-00667, joint CNRS-RFBR grant 09-01-93106-NCNIL, and by Federal Agency for Science and Innovations of Russian Federation under contract 14.740.11.0081.

\end{document}